\documentclass[11pt,a4paper]{paper}

\pdfoutput=1 

\usepackage[latin1]{inputenc}

\usepackage{url}
\usepackage{enumerate}
\usepackage{pdflscape}

\usepackage{multirow,tabularx}
\usepackage{subfigure}
\usepackage{multicol}
\usepackage{booktabs}
\usepackage{chngpage} 

\usepackage{graphicx,color}

\usepackage{amsmath,amssymb}
\usepackage{amsfonts}

\numberwithin{equation}{section}

\usepackage{threeparttable}

\usepackage{setspace}
\newtheorem{theorem}{Theorem}[section]

\newtheorem{proposition}[theorem]{\bf Proposition}
\newtheorem{corollary}[theorem]{\bf Corollary}

\newenvironment{proof}[1][Proof.]{\begin{trivlist}
\item[\hskip \labelsep {\bfseries #1}]}{\end{trivlist}}

\title{{\small Working paper}\\
A generalized formulation for vehicle routing problems}
\author{Pedro Munari$^a$, \ Twan Dollevoet$^b$, \ Remy Spliet$^b$ \\
\small $^a$Production Engineering Department, Federal University of S\~ao Carlos, Brazil\\
\small $^{b}$Econometric Institute, Erasmus University Rotterdam, The Netherlands\\
\small munari@dep.ufscar.br, \ dollevoet@ese.eur.nl, \ spliet@ese.eur.nl\\
\small September, 2017}

\begin{document}


\allowdisplaybreaks

\maketitle

\begin{abstract} 
Different types of formulations are proposed in the literature to model vehicle routing problems. Currently, the most used ones can be fitted into two classes, namely vehicle flow formulations and set partitioning formulations. These types of formulations differ from each other not only due to their variables and constraints but also due to their main features. Vehicle flow formulations have the advantage of being compact models, so general-purpose optimization packages can be used to straightforwardly solve them. However, they typically show weak linear relaxations and have a large number of constraints. Branch-and-cut methods based on specialized valid inequalities can also be devised to solve these formulations, but they have not shown to be effective for large-scale instances. On the other hand, set partitioning formulations have stronger linear relaxations, but requires the implementation of sophisticate techniques such as column generation and specialized branch-and-price methods. Due to all these reasons, so far it is has been recognized in the vehicle routing community that these two types of formulations are rather different. In this paper, we show that they are actually strongly related as they correspond to special cases of a generalized formulation of vehicle routing problems. 
\end{abstract} 

\section{Introduction}

The literature on vehicle routing problems has become very rich and covers nowadays a variety of applications, modeling approaches and solution methods \cite{toth2014}. Due to their huge importance in practice, these problems have called the attention of many researchers and motivated a large number of collaborations between companies and academia \cite{golden2008}. In addition, vehicle routing problems lead to challenging formulations that require the development of sophisticate solution strategies and motivates the design of clever heuristics and meta-heuristics \cite{baldacci2012,laporte2013}.

Vehicle routing problems are typically modeled using two different types of formulations. The first type, known as vehicle flow (VF) formulation, is based on binary variables associated to arcs of a network representation of the problem. In general, this is more intuitive and leads to a compact model that can be straightforwardly put on a black-box optimization solver. Also, valid inequalities and constraints (most of them exponential in terms of the number of customers) have been used to achieve a more effective strategy, resulting in specialized branch-and-cut methods. However, even with the use of very elaborate inequalities, VF formulations may be still very challenging for current optimization solvers. The main reason is the weak linear relaxation of these formulations.

A stronger linear relaxation is observed in the second type of models, known as set partitioning (SP) formulation. The number of constraints in this formulation is much smaller with respect to a VF formulation, but it has a huge number of variables: one for each feasible route in the problem. In the vast majority of cases, generating all these routes is not viable and hence the column generation technique is required to generate columns in an iterative way. Columns correspond to an incidence vector of feasible routes, which are generated by solving a resource constrained shortest path problem (RCESPP). Most implementations solve the RCESPP by a label-setting algorithm, which is aided with clever strategies to improve its performance \cite{martinelli2014, contardo2015}. The solution strategies based on SP formulations are currently the most efficient to obtain optimal solutions of vehicle routing problems \cite{baldacci2012,pecin2016}. Still, the performance can be very poor on problems that allow long routes, i.e. routes that visit many customers.
 
From this brief description of the two most used types of VRP formulations, we can observe that they have many opposite features and then can be recognized as very different from each other. However, in this paper we show that they are not so different, as they are actually special cases of a general formulation of vehicle routing problems, which we call as $p$-step formulation. In fact, this is a family of formulations, as different values of $p$ lead to different formulations. We show that the VF formulation and the SP formulation are $p$-step formulations with particular choices of $p$. In addition, we prove a relationship between the bounds provided by the linear relaxation of $p$-step formulations with different $p$. Column generation can also be used to solve a $p$-step formulation, with the advantage that more dual information is sent to the RCESPP with respect to SP formulations.

The $p$-step formulation associate variables to partial paths in the network representation of the problem. This has the potential of reducing the difficulty of solving problems that allow long routes, the big challenge in a SP formulation. On the other hand, $p$-step formulations may lead to stronger linear relaxations than a VF formulation, the main weakness of this latter. Many other advantages can be achieved by using a $p$-step formulation, as we enunciate ahead in this paper.

A formulation based on partial paths has also been proposed in \cite{petersen2009}, for the VRP with time windows. Similarly to the $p$-step formulation, the partial paths can start and end at any node of the network and must visit exactly a given number of customers. The authors obtain this formulation by applying Dantzig-Wolfe decomposition to a modified vehicle flow formulation of the problem, which relies on a modified graph to represent the solution as a giant tour. They prove that the linear relaxation of the resulting model provides a bound that is larger than or equal to the bound provided by the standard two-index flow formulation. The relationship between formulations with different resource bounds is not analyzed by the authors and no computational experiments are reported for the proposed formulation.


A similar idea has also been applied to other types of problem. In \cite{fragkos2016}, the authors propose a formulation based on horizon decomposition for the capacitated lot sizing problem with setup times. They partition the time horizon in several subsets, possibly with overlap, to have smaller subproblems, so they can be quickly solved by a black-box optimization solver. In the column generation framework, columns become associated to production plans defined for only one of the partitions of the time horizon. These partial production plans are then combined in the master problems, as in the p-step formulation.

The remainder of this paper has the following structure. In Section \ref{sec:classical}, we review the vehicle flow and set partitioning formulations and quickly discuss about their main features. In Section \ref{sec:family}, we propose the $p$-step formulation and present theoretical results that relate the formulations obtained using different choices of $p$. The column generation scheme for $p$-step formulations is proposed in Section \ref{sec:cg:pstep}, followed by a discussion of its main advantages and disadvantages in Section \ref{sec:advantages}.

\section{Classical VRP formulations}\label{sec:classical}

In this section, we review the formulations of two classical VRP variants, namely the capacitated vehicle routing problem (CVRP) and the vehicle routing problem with time windows (VRPTW). These formulations are the basis for many other variants of the problem. The purpose is to set the notation, nomenclature and foundations for the remaining sections of this paper. 

Consider a set of customers represented by $\mathcal{C} = \{ 1, \ldots, n \}$, such that a positive demand is associated to each customer $i \in \mathcal{C}$. To service these customers, we have to design routes for a fleet with $K$ vehicles available in a single depot. 
Each route must start at the depot, visit a subset of customers and then return to the depot. All customers must be visited exactly once.
Each vehicle has a maximum capacity $Q$, which limits the number of customers it can visit before returning to the depot. For the sake of clarity, we assume a homogeneous fleet of vehicles, but the discussion presented ahead can be easily extended to a heterogeneous fleet.

We represent the problem using a graph $G(\mathcal{N},\mathcal{E})$, in which $\mathcal{N} = \mathcal{C} \cup \{0, n+1 \}$ is the set of nodes associated to customers in $\mathcal{C}$ and to the depot nodes $0$ and $n+1$. We use two nodes to represent the same single depot and impose that all routes must start on $0$ and return to $n+1$. Set $\mathcal{E}$ contains the arcs $(i,j)$ for each pair of nodes $i,j \in \mathcal{N}$ (we assume a complete graph). The cost of crossing an arc $(i,j) \in \mathcal{E}$ is denoted by $c_{ij}$. Each node has a demand $q_i$, such that $q_i > 0$ for each $i \in \mathcal{C}$ and $q_0 = q_{n+1} = 0$. The objective of the problem is to determine a set of minimal cost routes that satisfies all the requirements defined above. 

\subsection{Two-index vehicle flow formulation}\label{sec:det:twoindex}

In the two-index VF formulation, we define the binary decision variable $x_{ij}$ that assumes value $1$ if and only if there is a route that goes from customer $i$ to $j$ directly, for $i, j \in \mathcal{N}$. In addition, $y_{j}$ is a continuous decision variable corresponding to the cumulated demand on the route that visits node $j \in \mathcal{N}$ up to this visit. With these parameters and decision variables, the two-index flow formulation of the CVRP if given by:
\begin{eqnarray}
\mbox{min} & \displaystyle \sum_{i=0}^{n+1} \sum_{j=0}^{n+1} c_{ij} x_{ij} &  \label{eq:vrptwarc2:obj} \\
\mbox{s.t.} & \displaystyle \sum_{j=1 \atop j \neq i}^{n+1}  x_{ij} = 1,& \ \ i = 1, \ldots, n, \label{eq:vrptwarc2:1} \\
& \displaystyle \sum_{i=0 \atop i \neq h}^{n} x_{ih} - \sum_{j=1 \atop j \neq h}^{n+1} x_{hj} = 0, & \ \ h = 1, \ldots, n, \label{eq:vrptwarc2:2} \\
& \displaystyle \sum_{j=1}^{n}  x_{0j} \leq K, &  \label{eq:vrptwarc2:10} \\
& \displaystyle y_{j} \geq y_{i} + q_j x_{ij} - Q ( 1 - x_{ij} ), & \ \ i,j = 0, \ldots, n+1, \label{eq:vrptwarc2:4} \\
& \displaystyle d_{i} \leq y_{i} \leq Q, & \ \ i = 0, \ldots, n+1, \label{eq:vrptwarc2:6} \\
& x_{ij} \in \{0, 1\}, & \ \  i,j = 0, \ldots, n+1. \label{eq:vrptwarc2:8}
\end{eqnarray}
Constraints \eqref{eq:vrptwarc2:1} ensure that all customers are visited exactly once. Constraints \eqref{eq:vrptwarc2:2} guarantee the correct flow of vehicles through the arcs, by stating that if a vehicle arrives to a node $h \in \mathcal{N}$, then it must depart from this node. Constraint  \eqref{eq:vrptwarc2:10} limits the maximum number of routes to $K$, the number of vehicles.
Constraints \eqref{eq:vrptwarc2:4} and \eqref{eq:vrptwarc2:6} ensure together that the vehicle capacity is not exceeded. The objective function is defined by \eqref{eq:vrptwarc2:obj} and imposes that the total travel cost of the routes is minimized.

Constraints \eqref{eq:vrptwarc2:4} also avoid subtours in the solution, \textit{i.e.} cycling routes that do not pass through the depot. Different types of constraints are proposed in the literature to impose vehicle capacities and/or avoid subtours \cite{irnich2014}. The advantage of using \eqref{eq:vrptwarc2:4} and \eqref{eq:vrptwarc2:6} is that the model has a polynomial number of constraints in terms of the number of customers. However, the lower bound provided by the linear relaxation of this model is known to be weak in relation to other models. Hence, many authors recur to capacity constraints that results in better lower bounds, even though the number of constraints becomes exponential in terms of the number of customers, requiring the use of a branch-and-cut strategy \cite{semet2014}. 


The VRPTW is an entension of the CVRP, in which customer time windows are imposed for the visits. A time window corresponds to a time interval $[w^{a}_{i}, w^{b}_{i}]$ which imposes that the service at node $i \in \mathcal{N}$ cannot start earlier than the time instant $w^{a}_{i}$ nor later than $w^{b}_{i}$. If the vehicle arrives before than $w^{a}_{i}$, then it has to wait until this instant to start servicing the node. To each arc $(i,j) \in \mathcal{E}$, we assign a travel time $t_{ij}$, which respects triangle inequality. Also, each node $i$ has a service time $t_i$ that corresponds to the minimum amount of time that the vehicle has to stay in a visited node. 

Let $w_{i}$ be a continuous decision variable representing the time instant that the service starts at node $i \in \mathcal{N}$. We obtain a model for the VRPTW by adding the following constraints to the formulation \eqref{eq:vrptwarc2:obj}--\eqref{eq:vrptwarc2:8}:
\begin{eqnarray}
& w_{j} \geq w_{i} + ( s_i + t_{ij} ) x_{ij} - M_{ij} ( 1 - x_{ij} ), & \ \ i = 0, \ldots, n; \ \ j = 1, \ldots, n+1, \label{eq:vrptwarc2:11} \\ 
& w^a_i \leq w_{i} \leq w^b_i, & \ \ i = 0, \ldots, n+1, \label{eq:vrptwarc2:12} 
\end{eqnarray}
where $M_{ij}$ is a sufficiently large value, which can be defined as $M_{ij} = \max\{ w^b_i - w^a_j, 0 \}$. 

\subsection{Set partitioning formulation}\label{sec:det:setpart}


Currently, the most efficient exact methods for solving VRP variants are based on SP formulations. The variables in these formulations correspond to feasible routes of the problem. Let $\mathcal{R}$ be the set of routes that satisfy the problem requirements. For example, in the CVRP, a route in $\mathcal{R}$ must start and finish at the depot, visit at most once a customer, respect the vehicle capacity and guarantee that if the route arrives to a customer than it has to leave this customer. The same requirements are valid for the VRPTW, in addition to satisfying time windows of all visited customers.

Let $\lambda_r$ be the binary decision variable that is equal to 1 if and only if the route $r \in \mathcal{R}$ is selected. The SP formulation is as follows:
\begin{eqnarray}
\mbox{min} & \displaystyle \sum_{r \in \mathcal{R}} c_{r} \lambda_r &  \label{eq:vrp:sp:obj} \\
\mbox{s.t.} & \displaystyle \sum_{r \in \mathcal{R}} a_{ri} \lambda_r = 1, & i \in \mathcal{C}, \label{eq:vrp:sp:1} \\
 & \displaystyle \sum_{r \in \mathcal{R}} \lambda_r \leq K, & \label{eq:vrp:sp:2} \\
 & \displaystyle \lambda_r  \in \{0, 1\}, & r \in \mathcal{R}. \label{eq:vrp:sp:3}
\end{eqnarray}
This can be used to model the CVRP, VRPTW and many other VRP variants, depending on how we define the set of routes $\mathcal{R}$. The objective function \eqref{eq:vrp:sp:obj} minimizes the total cost of the selected routes. The cost of route $r \in \mathcal{R}$, denoted by $c_{r}$, is computed using the arc costs $c_{ij}$ defined above. Namely, given a route $r$ that sequentially visits nodes $i_0, i_1, \ldots, i_p$, $p>0$, its total cost is given by
\begin{equation}
c_{r} = \sum_{j = 0}^{p-1} c_{ i_{j} i_{j+1} },
\end{equation}
Constraints \eqref{eq:vrp:sp:1} impose exactly one visit to each customer node. Each column $a_{r} = (a_{r1}, \ldots, a_{rn})^{T}$ is a binary vector in which $a_{ri} = 1$ if and only if the corresponding route $r$ visits customer $i$. Constraint \eqref{eq:vrp:sp:2} imposes the maximum number of vehicles available at the depot. If $K$ is sufficiently large for the problem, than this constraint can be dropped from the formulation. 

Generating all routes of $\mathcal{R}$ is impractical in general, as the number of routes is exponential in terms of the numbers of customers. Hence, set partitioning formulations require using the column generation technique for solving the linear relaxation of model \eqref{eq:vrp:sp:obj}--\eqref{eq:vrp:sp:3} \cite{lubbecke2005}. As a consequence, to obtain optimal integer solutions we need a branch-and-price method \cite{poggi2014}. In the column generation technique, we start with a small subset of routes $\overline{\mathcal{R}} \subset \mathcal{R}$ that is used to create the following restricted master problem (RMP):
\begin{eqnarray}
\mbox{min} & \displaystyle \sum_{r \in \overline{\mathcal{R}}} c_{r} \lambda_r &  \label{eq:vrp:rmp:obj} \\
\mbox{s.t.} & \displaystyle \sum_{r \in \overline{\mathcal{R}}} a_{ri} \lambda_r = 1, & i \in \mathcal{C}, \label{eq:vrp:rmp:1} \\
 & \displaystyle \sum_{r \in \overline{\mathcal{R}}} \lambda_r \leq K, & \label{eq:vrp:rmp:2} \\
 & \displaystyle \lambda_r  \geq 0, & r \in \overline{\mathcal{R}}. \label{eq:vrp:rmp:3}
\end{eqnarray}
Notice that the RMP is the linear relaxation of \eqref{eq:vrp:sp:obj}--\eqref{eq:vrp:sp:3}, but considering only a subset of variables. Let $u = (u_1, \ldots, u_n) \in \mathbb{R}^{n}$ and $\sigma \in \mathbb{R}_{-}$ be the dual variables associated to constraints \eqref{eq:vrp:rmp:1} and \eqref{eq:vrp:rmp:2}, respectively. At each iteration of the column generation method, we solve the RMP to obtain a dual solution $(\overline{u},\overline{\sigma})$ that is used to generate the columns that are not in the RMP yet. These columns are associated with feasible routes obtained by solving the following subproblem:
\begin{equation}
\displaystyle \min_{r \in \mathcal{R}} rc(\overline{u},\overline{\sigma}) = \sum_{i \in \mathcal{N}} \sum_{j \in \mathcal{N}} (c_{ij} - \overline{u}_i^{ }) x_{rij} - \overline{\sigma}  \label{eq:subproblem}
\end{equation}
where $\overline{u}_0 = \overline{u}_{n+1} = 0$ and  $x_{r} = \{ x_{rij} \}_{i,j \in \mathcal{N}}$ is a binary vector such that $x_{rij} = 1$ if and only if route $r \in \mathcal{R}$ visits node $i$ and goes directly to node $j$. This subproblem is a \textit{resource constrained elementary shortest path problems} (RCESPP) \cite{irnich2005}.
Let $\bar{x}_{r}$ be associated to an optimal route $r$ of the subproblem. If the corresponding value $rc(\overline{u},\overline{\sigma})$ is negative, then a new variable $\lambda_r$ can be added to the RMP using this route. Indeed, $rc(\overline{u},\overline{\sigma})$ is the reduced cost of this new variable, for which we have the following cost and column coefficients:
\begin{equation}
c_r := \sum_{i\in \mathcal{N}} \sum_{j\in \mathcal{N}} c_{ij}^{ } \bar{x}_{rij}, \nonumber
\end{equation}
\begin{equation}
a_{ri} := \sum_{j \in \mathcal{N}} \bar{x}_{rij}, \ i \in \mathcal{C}. \nonumber
\end{equation}
Hence, $r$ is added to $\overline{\mathcal{R}}$ and the new RMP has to be solved again. If $rc(\overline{u},\overline{\sigma})$ is nonnegative and $(\overline{u},\overline{\sigma})$ are optimal dual solutions of the current RMP, then the optimal solution of the current RMP is also optimal for the linear relaxation of the MP. Hence, the column generation method terminates successfully. 

The performance of a computational implementation of the column generation algorithm is strongly dependent on the way RMPs and subproblems are solved. To be successful, implementations should quickly solve the RMPs and use stable dual solutions that help to reduce the total number of iterations \cite{lubbecke2005,munari2013,munari2015}. Solving the RCESPP effectively is also a very important requirement in a column generation algorithm for VRP variants. Although integer programming formulations are available for the RCESPP, they cannot be solved effectively by the current state-of-the-art optimization solvers \cite{pugliese2013}. The current best strategies use a label-setting algorithm based on dynamic programming. This algorithm was originally proposed by \cite{desrochers1988} and \cite{beasley1989} and since then has been continuously improved \cite{feillet2004, righini2008, chabrier2006, desaulniers2008, baldacci2011, martinelli2014, contardo2015}.


\section{A family of vehicle routing problem formulations}\label{sec:family}

In this section, we propose a generalized VRP formulation. The idea of this new family of formulations is to have binary variables associated to up to $p$ (sequential) steps in the network. One step corresponds to traversing a given arc in the network, so $p$ steps correspond to a partial path that traverses exactly $p$ arcs. Let $\mathcal{S}$ be the set of all feasible $p$-steps in the network, including also all the feasible $k$-step paths that start at node $0$, for all $k = 1, \ldots, p-1$, when $p > 1$. By feasible we mean that the arcs can be traversed sequentially and that none of the resources are violated. For instance, if $p=2$ then $\mathcal{S}$ is the set of all partial paths of the forms $i-j-k$ and $0-j$, for any $i, j, k \in \mathcal{N}$.
Given a partial $k$-step path $s \in \mathcal{S}$, for $k = 1, \ldots, p$, we denote by $i_s$ and $j_s$ its first and last nodes, respectively. Let $\lambda_{s}$ be a binary variable that is equal to $1$ if and only if the arcs in $s \in \mathcal{S}$ are traversed sequentially in the optimal path. For the sake of clarity, we assume at first that capacity is the only resource in the problem. Let $\varphi_j$ be a continuous decision variable that is equal to the cumulated demand of all nodes visited by a route up to node $j$ (inclusive). The $p$-step formulation for the CVRP is as follows:
\begin{eqnarray}
\mbox{min} & \displaystyle  \sum_{s \in \mathcal{S}} c_{s} \lambda_{s} &  \label{eq:vrp:pstep:obj} \\
\mbox{s.t.} & \displaystyle  \sum_{s \in \mathcal{S}} e_{i}^{s} \lambda_{s} = 1, & i = 1, \ldots, n, \label{eq:vrp:pstep:2} \\
& \displaystyle  \sum_{s \in \mathcal{S}} a_{i}^{s} \lambda_{s} = 0, & i = 1, \ldots, n, \label{eq:vrp:pstep:1a} \\
& \displaystyle  \sum_{s \in \mathcal{S}} a_{0}^{s} \lambda_{s} \leq K, &  \label{eq:vrp:pstep:1b} \\
& \displaystyle \varphi_{j} \geq \varphi_{i} + q_j \sum_{s \in \mathcal{S}_{ij}} \lambda_{s} - Q (1 - \sum_{s \in \mathcal{S}_{ij}} \lambda_s), & i = 0, \ldots, n, \ j = 1, \ldots, n+1, \label{eq:vrp:pstep:3}\\
& \displaystyle q_{i} \leq \varphi_{i} \leq Q, &  i = 1, \ldots, n, \label{eq:vrp:pstep:4} \\
& \displaystyle \lambda_s  \in \{0, 1\}, & s \in \mathcal{S}, \label{eq:vrp:pstep:5}
\end{eqnarray}
where $\mathcal{S}_{ij} \subset \mathcal{S}$ contains only the paths that traverse arc $(i,j)$ for a given pair $i,j \in \mathcal{N}$; $c_{s}$ is the total cost of traversing all arcs in path $s$; and $a^{s}$ and $e^{s}$ are vectors defined as 
\begin{equation}
a^{s}_i = \left\{
	\begin{array}{rl}
		+1,& \mbox{if } i \mbox{ is the first node visited by path } s,\\
		-1,& \mbox{if } i \mbox{ is the last node visited by path } s,\\
		0, & \mbox{otherwise},
	\end{array} \right. 
\end{equation}
\begin{equation}
e^{s}_i = \left\{
	\begin{array}{rl}
		+1,& \mbox{if } i \mbox{ is visited by path } s\mbox{, but it is not the last node of } s, \\
		0, & \mbox{otherwise},
	\end{array} \right. 
\end{equation}
for all $i = 0, 1, \ldots, n$ and  $s \in \mathcal{S}$. In this formulation, constraints \eqref{eq:vrp:pstep:2} impose that each customer node is visited at most once; constraints \eqref{eq:vrp:pstep:1a} ensure that if two paths are linked, then the last node in one path is the same as the first node in the other; constraint \eqref{eq:vrp:pstep:1b} imposes the maximum number of (complete) routes in an optimal solution; constraints \eqref{eq:vrp:pstep:3} and \eqref{eq:vrp:pstep:4} ensure that routes satisfy the capacity resource and has no subtours; and \eqref{eq:vrp:pstep:5} impose the binary domain of the decision variables $\lambda$. Notice that the $p$-step paths in $\mathcal{S}$ must traverse exactly $p$ arcs in the network. The only partial paths in $\mathcal{S}$ that are allowed to cross less than $p$ arcs are those that starts on the depot node 0.

Customer time windows can also be included in the $p$-step formulation by adding the following constraints, resulting in the VRPTW $p$-step model:
\begin{equation}
\omega_{j} \geq \omega_{i} + ( s_i + t_{ij} ) \sum_{s \in \mathcal{S}_{ij}} \lambda_s - M ( 1 - \sum_{s \in \mathcal{S}_{ij}} \lambda_s ), \  i = 0, \ldots, n, \ j = 1, \ldots, n+1,
\end{equation}
\begin{equation}
w^a_i \leq \omega_{i} \leq w^b_i, \ \ i = 0, \ldots, n+1, \nonumber
\end{equation}
where $\omega_i$ is a decision variable that indicates the time instant that the service starts at node $i = 0, \ldots, n+1$. 

Constraints \eqref{eq:vrp:pstep:3} can be written in a coupled way, based only on the first ($i_s$) and last ($j_s$) nodes of the partial path $s \in \mathcal{S}$, as follows:
\begin{equation}
\displaystyle \varphi_{j_s} \geq \varphi_{i_s} + q_s \lambda_{s} - Q (1 - \lambda_s), \ s \in \mathcal{S},
\end{equation}
\begin{equation}
\displaystyle q_{s} \leq \varphi_{i_s} \leq Q, s \in \mathcal{S}.
\end{equation}
where $q_{s}$ is the total demand of the nodes visited by this path, except for its first node. 

\subsection{Special cases of the $p$-step family}

As mentioned before, model \eqref{eq:vrp:pstep:obj}--\eqref{eq:vrp:pstep:5} is a family of formulations, because for each $p = 1, \ldots, n$ we have different paths. Indeed, for particular choices of $p$, we obtain the VF formulation \eqref{eq:vrptwarc2:obj}--\eqref{eq:vrptwarc2:8} and the SP formulation \eqref{eq:vrp:sp:obj}--\eqref{eq:vrp:sp:3}, as presented in Proposition \ref{prop:particularcases}.

\begin{proposition}\label{prop:particularcases}
The vehicle flow formulation \eqref{eq:vrptwarc2:obj}--\eqref{eq:vrptwarc2:8} and the set partitioning formulation \eqref{eq:vrp:sp:obj}--\eqref{eq:vrp:sp:3} are particular cases of the $p$-step formulation, with $p=1$ and $p=n+1$, respectively.
\end{proposition}
\begin{proof}
In the $1$-step formulation, the set $\mathcal{S}$ is given by all the single arcs $(i,j)$, with $i,j = 0, 1, \ldots, n+1$. Hence, all variables in this formulation can be rewritten as $\lambda^s = x_{i_s j_s}$, where $(i_s, j_s)$ is the arc traversed by the 1-step path $s \in \mathcal{S}$. By replacing this in the $p$-step formulation \eqref{eq:vrp:pstep:obj}--\eqref{eq:vrp:pstep:5} and noticing that path $s$ can be expressed uniquely by its corresponding pair of nodes $i_s$ and $j_s$, we obtain the VF formulation \eqref{eq:vrptwarc2:obj}--\eqref{eq:vrptwarc2:8}. 

On the other hand, in the $(n+1)$-step formulation, set $\mathcal{S}$ is given by all the $k$-paths that start at node $0$, for $1 < k < n+1$ and all $(n+1)$-steps that start at $0$ end at $n+1$. Hence, $\mathcal{S}$ can be reduced to all feasible complete routes, as in the usual set partitioning formulation. However, the $p$-step formulation contains more variables and more constraints. In order to show that the formulations are equivalent, we show that feasible variables $\phi_i$ can always be chosen. In order to so, let a solution to the set partitioning formulation be given. By definition, every route starts and ends at the depot. Every route corresponds to an $(n+1)$-step or to a $k$-step, with $1\leq k < n+1$. It holds that $e_i^s=1$ if $i$ is visited in the $(n+1)$-step or $k$-step $s$ and $e_i^s = 0$ otherwise, for $1\leq i\leq n$. Defining $\mathcal{S}_i =\{s \in \mathcal{S}: i\in s\}$, Constraint \eqref{eq:vrp:sp:1} implies, for all $1\leq i \leq n$, that
$$1 = \sum_{s \in \mathcal{S}} \lambda_s e_i^s =  \sum_{s\in \mathcal{S}_i} \lambda_s.$$ 
Recall that $\mathcal{S}_{ij} = \{s\in \mathcal{S}_i: (i,j)\in s\}$ and define $\mathcal{S}_i'=\mathcal{S}_i \setminus \mathcal{S}_{ij}$. Then $\mathcal{S}_{ij}$ contains all $k$-steps and $(n+1)$-steps that visit $j$ directly after visiting $i$ and $\mathcal{S}_i'$ contains all other $k$-steps or $(n+1)$-steps that visit $i$. The above implies that
$$1 = \sum_{s \in \mathcal{S}_{ij}} \lambda_s + \sum_{s\in \mathcal{S}_i'}\lambda_s.$$
For each $s\in \mathcal{S}$, we define, for all $1\leq i\leq n$, using the topological ordering $\leq_s$ induced by $s$,
$$\phi_i^s = \left\{\begin{array}{ll}\sum_{i'\in s, i' \leq_s i}q_{i'} &\textrm{if }i \in s, \\0 & \textrm{if }i \notin s.\end{array} \right.$$
as the cumulated demand on $s$ up until node $i$. By construction, $0 \leq \phi_i^s \leq Q$. This implies that $\phi_j^s -\phi_i^s \geq -Q$. Define now
$$\phi_i = \sum_{s\in\mathcal{S}} \lambda_s \phi_i^s = \sum_{s \in \mathcal{S}_i} \lambda_s \phi_i^s.$$ 
If $i$ is visited in $s$, then $q_i \leq \phi_i^s$. In that case $q_i \leq \phi_i^s\leq Q$. Taking a convex combination of these inequalities, we obtain
$$q_i = \sum_{s \in \mathcal{S}_i} \lambda_s q_i  \leq \sum_{s\in \mathcal{S}_i}\lambda_s\phi_i^s \leq \sum_{s\in \mathcal{S}_i} \lambda_sQ =Q.$$
By definition of $\phi_i$, Constraints \eqref{eq:vrp:pstep:4} are satisfied for all $1\leq i \leq n$. 
Furthermore, for any arc $(i,j)$ it holds that 
\begin{align*}
\phi_j - \phi_i &= \sum_{s \in \mathcal{S}} \lambda_s (\phi_j^s -\phi_i^s) \\
&\geq \sum_{s \in \mathcal{S}_i} \lambda_s (\phi_j^s - \phi_i^s)\\
&= \sum_{s\in \mathcal{S}_{ij}} \lambda_s (\phi_j^s - \phi_i^s) + \sum_{s\in \mathcal{S}_i'} \lambda_s (\phi_j^s - \phi_i^s)\\
&\geq \sum_{s\in \mathcal{S}_{ij}} \lambda_s q_j + \sum_{s\in \mathcal{S}_{i}'} \lambda_s (-Q) = q_j \sum_{s\in \mathcal{S}_{ij}} \lambda_s   -Q \left(1 - \sum_{s\in \mathcal{S}_{ij}} \lambda_s\right) 
\end{align*}
This shows that \eqref{eq:vrp:pstep:3} holds for all $0\leq i\leq n$ and $1\leq j\leq n+1$ as well. We conclude that the $p$-step formulation for $p=n+1$ and the set partitioning formulation are equivalent. \hfill$\square$

\end{proof}

\subsection{Intermediate $p$-step formulations}

For any value of $p = 1, \ldots, n+1$, the corresponding $p$-step formulation is a valid vehicle routing problem formulation. 
The basic difference between formulations with different values of $p$ lies in the level of arc coupling in the partial paths. Indeed, in the $1$-step formulation, the arcs are totally detached, so the model has to decide what is the best way of connecting them, without violating other constraints such as elementarity and resource availability. The number of variables in the model is polynomial in terms of the number of nodes, so a general-purpose optimization package could be used to solve it. In addition, generating these paths is quick and straightforward. However, the VF formulation is well known for its poor performance, mainly due to a weak linear relaxation. On the other, in the $(n+1)$-step formulation all the arcs are already attached so the model has only to choose what is the best set of routes. The SP formulation is well recognized by having a stronger linear relaxation, but column generation and branch-and-price methods are required to solve the problem, as the number of variables is exponential in terms of the number of customers. In this case, the difficulty lies in generating the paths, as they must be feasible routes that depart from and return to the depot. These features illustrate that VF formulations and SP formulations are extremal cases of $p$-step formulations. 

At this point, an intriguing question emerges: Is there a choice of $p$ such that the $p$-step formulation has a reasonably strong linear relaxation and performs well in practice? Proposition \ref{prop:particularcasesobj_v2} brings an interesting relationship between $p$-step formulations with different values of $p$, regarding the optimal values of their respective linear relaxations.

\begin{proposition}\label{prop:particularcasesobj_v2}
Let $\tilde{z}_p$ be the optimal value of the linear relaxation of a $p$-step formulation, for $p = 1, \ldots, n+1$. For any $p \in \{1, \ldots, n\}$ and $q\geq 2$ such that $pq \leq n+1$, we have that $\tilde{z}_{pq} \geq \tilde{z}_{p}$.
\end{proposition}
\begin{proof}
Consider the optimal solution using $pq$-steps. This solution selects a set of $pq$-steps and $k$-steps starting at 0, for $1\leq k < pq$, with corresponding $\lambda^1$ and $\phi$ variables. Any given $pq$-step $s$ can be cut into exactly $q$ $p$-steps $s_1,\ldots,s_q$. We define $\lambda^2_{s_i} = \lambda^1_s$ for all $1\leq i \leq q$. Consider any $k$-step $s$ that start at 0, with $k<pq$. Then $k$ can be written as $k=pq'+ k'$, with $0\leq k'<p$ and $0\leq q'<q$. We can cut this $k$-step into $q'$ $p$-steps $s_1,\ldots,s_{q'}$, and, if $k'\neq 0$, one $k'$-step $s'$. We define $\lambda^2_{s_i} = \lambda^1_s$ for all $1\leq i\leq q'$ and $\lambda^2_{s'}=\lambda^1_s$ (if $k'\neq 0$). It follows easily that $(\lambda^2,\phi)$ gives a feasible solution with $p$-steps with the same objective. This shows that $\tilde{z}_p \leq \tilde{z}_{pq}$.\hfill $\square$
\end{proof}

Using $p=1$ and $q=p$ in the above theorem, it follows that any formulation is at least as strong as the vehicle flow formulation.
\begin{corollary}
For any $1 < p \leq n+ 1$, it holds that $ \tilde{z}_{p} \geq \tilde{z}_1$.
\end{corollary}

Furthermore, no formulation is strictly stronger than the set partitioning formulation.
\begin{proposition}
For any $1 \leq p < n+1$, it holds that $ \tilde{z}_{n+1} \geq \tilde{z}_p$.
\end{proposition}

\begin{proof}
We use a similar argument as in the proof of Proposition~\ref{prop:particularcasesobj_v2}. Any $(n+1)$-step or $k$-step that starts at 0, with $1\leq k < n+1$, selected in the set-partitioning formulation, can be cut into $p$-steps and $k'$-steps, with $1\leq k' < p$. This shows that $\tilde{z}_{n+1}\geq \tilde{z}_p$. \hfill$\square$.
\end{proof}

In the remainder of this section, we show that one cannot compare the formulations for $p\neq 1$ and $p'\neq n+1$ in general, if $p' > p$ is not a multiple of $p$.

\newcommand{\e}{e}

\begin{proposition}\label{prop:notMultiple1}
Let $p\geq 2$, $q\geq 1$ and $1\leq k < p$ be given. Define $p'=q\cdot p + k$. There exists an instance of the CVRP for which the \emph{strict} inequality
$$\tilde{z}_{p'} < \tilde{z}_p$$
holds.
\end{proposition}

\begin{proof}
Define $n=(q+1)\cdot p$ and assume that the nodes are clustered: they appear in $m=q+1$ groups of $p$ nodes. Within a cluster, the distance is negligible. The clusters themselves are located on the vertices of a regular convex $m$-polygon. The length of an edge of the polygon is normalized to 1. The depot is located far away from all nodes: The distance from the depot to each node is larger than 1. The capacity of a vehicle is equal to $n$. All nodes have unit demand.\\

We now construct a feasible solution that uses $p'$-steps. In order to so, we first define a set of $(p-1)$-steps. For each cluster $1\leq c\leq m$ we consider the so-called \emph{regular} $(p-1)$-step 
$$(c-1)\cdot p + 1 \rightarrow \ldots \rightarrow (c-1)\cdot p + p.$$
We can cyclically permute these $(p-1)$-steps. We denote $P_c^t$ as the above $(p-1)$-step that is cyclically permuted $t$ times, for $0\leq t < p$. Formally, it is defined by
$$P_c^t = (c-1)\cdot p + [1 + t]_p \rightarrow \cdots \rightarrow (c-1)\cdot p + [p + t]_p.$$
Here, we denote $[a]_b$ for $a\mod b$. We now define for every $(t,c) \in \{0,\ldots, p-1\}\times\{1,\ldots,m\}$, the following $(m\cdot p - 1)$-step
$$P_c^t \Rightarrow P_{[c+1]_m}^t \Rightarrow P_{[c+2]_m}^t \Rightarrow \ldots \Rightarrow P_{[c+m-1]_m}^t.$$
Here, arcs denoted by `$\Rightarrow$' have length 1 whereas arcs denoted by `$\rightarrow$' have negligible length. There are $p\cdot m =n$ such $(m\cdot p -1)$-steps.

By construction, for a given position between 1 and $n$, every node appears at that position in exactly one $(m\cdot p - 1)$-step. This also holds if we truncate all $(m\cdot p -1)$-steps after $p'$ steps, thereby obtaining $p'$-steps. This gives us a set of $n$ $p'$-steps. Selecting all these $p'$-steps with $\lambda=\frac{1}{p'}$ gives a feasible solution to the problem, if we also define $\phi_i = 1$ for all $i$. The non-trivial step is to show that Constraints~\eqref{eq:vrp:pstep:3} are satisfied. Note that any arc $(i,j)$ \emph{within} a cluster, satisfies
\begin{equation}\label{eq:flow}
\sum_{s \in \mathcal{S}_{ij}} \lambda_s \leq (q (p-1) + k)\cdot \frac{1}{p'} = \frac{p'- q}{p'}.
\end{equation}
(In particular, it is equal to the right hand side if $j = [i+1]_p$ and 0 otherwise). Arcs $(i,j)$ \emph{between} clusters satisfy
$$ \sum_{s \in \mathcal{S}_{ij}}\lambda_s\leq q \frac{1}{p'} = \frac{q}{p'}.$$
From $2q \leq pq < pq + k = p'$ we observe that $q < p'-q$. It follows that the flow over \emph{all} arcs $(i,j)$ is bounded by the right hand side of \eqref{eq:flow}. We now show that Constraints~\eqref{eq:vrp:pstep:3} are satisfied for this flow over arc $(i,j)$.
Given that $q\geq 1$, it follows that
$$p' < n < n+1 \leq q(n+1).$$
This shows that $p'-q < q n$. Dividing by $p'$, this yields
$$\frac{p'-q}{p'} < n \frac{q}{p'} = n \left(1 - \frac{p'-q}{p'}\right).$$
For all $0\leq i\leq n$ and $1\leq j \leq n+1$, it follows that
$$\phi_j - \phi_i = 0 > \frac{p'-q}{p'} - n\left(1 - \frac{p'- q}{p'}\right).$$
We conclude that the solution satisfies Constraints~\eqref{eq:vrp:pstep:3} if the flow over $(i,j)$ equals $\frac{p'-q}{p'}$. It follows easily that the constraints are also satisfied if the flow is smaller.

The distance traveled (so, the costs) for these $p'$-steps equals $m-1=q$. (By construction, only negiglible arcs are removed by the truncation.) The optimal objective of the LP-relaxation with $p'$-steps is at most the total costs of this feasible solution:
$$\tilde{z}_{p'} \leq n \frac{1}{p'} q = \frac{n\cdot q}{p'}.$$

Consider now any $p$-step $s$. Any $p$-step $s$ satisfies
$$\sum_{i=1}^n \e_{is} \leq p.$$
For any vector $\lambda_s$, we multiply this by $\lambda_s$, sum over $s$ and use \eqref{eq:vrp:pstep:2}. We then obtain
$$n = \sum_{i=1}^n \sum_{s\in \mathcal{S}} \lambda_s \e_{is}= \sum_{s\in \mathcal{S}} \sum_{i=1}^n \lambda_s \e_{is}=\sum_{s\in \mathcal{S}} \lambda_s \sum_{i=1}^n  \e_{is} \leq \sum_s \lambda_s p = p \sum_{s}\lambda_s.$$
The cost $c_s$ of a $p$-step is at least 1. Hence
$$\sum_{s\in \mathcal{S}} \lambda_s \leq \sum_s \lambda_s c_s.$$
Combining the above inequalities, we find
$$n \leq p \sum_{s\in \mathcal{S}} \lambda_s \leq p \sum_{s\in \mathcal{S}} c_s\lambda_s.$$
As this holds for every vector $\lambda_s$, it holds particularly for the optimal solution of the LP-relaxation with $p$-steps. Hence 
$$\tilde{z}_p = \sum_s \lambda_s c_s \geq \frac{n}{p}.$$
Combining the expressions for the values of the LP-relaxations, we obtain
$$\tilde{z}_{p'} \leq \frac{nq}{p'} = \frac{nq}{qp+k} < \frac{nq}{qp} = \frac{n}{p} \leq \tilde{z}_p.$$
This proves the claim.\hfill$\square$
\end{proof}

\begin{proposition} \label{prop:notMultiple2}
Let $p\geq 2$, $q\geq 1$ and $1\leq k < p$ be given. Define $p'=q\cdot p + k$. There exists an instance of the CVRP for which the \emph{strict} inequality
$$\tilde{z}_{p'} > \tilde{z}_p$$
holds.
\end{proposition}
\begin{proof}
Define $m \in \textbf{N}$ such that $n=m(p+1)\geq p'$. Consider $m$ clusters with $p+1$ nodes and a depot far away. Again, the distance within a cluster is negligible, while the distance between the clusters is normalized to 1. All customers have unit demand, and the capacity of the vehicle is equal to $n$.

We first generate a feasible solution using $p$-steps with negligible costs. Recall the definition of the $p$-steps $P^1_c,\ldots P^{p+1}_c$ for $1\leq c \leq m$ from the proof of Proposition~\ref{prop:notMultiple1} and also define $\bar{P}^t_c$ for $1\leq t\leq p+1$ and $1\leq c\leq m$ as the $p$-step $P^t_c$ in reverse order. This gives us $2m(p+1)$ $p$-steps. Define $\lambda_s=\frac{1}{2p}$ for all of them and $\phi_i=1$ for all $1\leq i\leq n$. By construction, the flow over any arc $(i,j)$ is at most $\frac12$. It follows that Constraints \eqref{eq:vrp:pstep:3} are satisfied. Thus, we defined a feasible solution with costs 0. It follows that $\tilde{z}_{p}\leq 0$. \\
By construction, any $p'$-step has strictly positive costs, as it uses at least one arc from one cluster to another or from the depot to a cluster. Hence, any $p'$-step $s$ satisfies $c_s\geq 1$. The inequality 
$$\sum_{i=1}^n \e_{is} \leq p'$$
now implies that
$$n  =\sum_{i=1}^n \e_{is} \lambda_s = \sum_{s\in \mathcal{S}} \lambda_s \sum_{i=1}^n \e_{is} \leq p'\sum_{s\in \mathcal{S}}\lambda_s \leq p'\sum_{s\in \mathcal{S}} \lambda_s c_s = p'\tilde{z}_{p'}.$$
We obtain 
$$\tilde{z}_{p'}\geq \frac{n}{p'} \geq 1 > \tilde{z}_p.$$ This proves the claim. \hfill$\square$
\end{proof}

\section{Column generation for the $p$-step formulations} \label{sec:cg:pstep}

Any $p$-step formulation can be seen as a column generation model, as any of its columns (variables) can be generated by following a known rule. Of course, for $p$ small, the number of columns in the formulation is polynomial in terms of the number of nodes and hence it can be practical to enumerate them beforehand. Even so, for large-scale problems it can be more advantageous to recur to column generation, as only a few variables will be nonzero at the optimal solution. 

Consider the linear relaxation of the $p$-step formulation \eqref{eq:vrp:pstep:obj}--\eqref{eq:vrp:pstep:5} having only the columns corresponding to an arbitrary subset $\overline{\mathcal{S}} \subset \mathcal{S}$. This leads to the following restricted master problem (RMP):
\begin{eqnarray}
\mbox{min} & \displaystyle  \sum_{s \in \overline{\mathcal{S}}} c^{s} \lambda^{s} &  \label{eq:vrp:pstep:rmp:obj} \\
\mbox{s.t.} & \displaystyle  \sum_{s \in \overline{\mathcal{S}}} e_{i}^{s} \lambda^{s} = 1, & i = 1, \ldots, n, \label{eq:vrp:pstep:rmp:2} \\
& \displaystyle  \sum_{s \in \overline{\mathcal{S}}} a_{i}^{s} \lambda^{s} = 0, & i = 1, \ldots, n, \label{eq:vrp:pstep:rmp:1a} \\
& \displaystyle  \sum_{s \in \overline{\mathcal{S}}} a_{0}^{s} \lambda^{s} \leq K, &  \label{eq:vrp:pstep:rmp:1b} \\
& \displaystyle \varphi_{i} - \varphi_{j} + \sum_{s \in \overline{\mathcal{S}}_{ij}} (q_j + Q) \lambda^{s} \leq Q, & i = 0, \ldots, n, \ j = 1, \ldots, n+1, \label{eq:vrp:pstep:rmp:3}\\
& \displaystyle q_{i} \leq \varphi_{i} \leq Q, &  i = 1, \ldots, n, \label{eq:vrp:pstep:rmp:4} \\
& \displaystyle \lambda_s  \geq 0, & s \in \overline{\mathcal{S}}, \label{eq:vrp:pstep:rmp:5}
\end{eqnarray}
where $\overline{\mathcal{S}}_{ij}$ has the same meaning as ${\mathcal{S}}_{ij}$, but considers only the paths in $\overline{\mathcal{S}}$.
Notice that we have written constraints \eqref{eq:vrp:pstep:rmp:3}--\eqref{eq:vrp:pstep:rmp:5} in a slightly different way for the sake of clarity. 

Let $u^1 = (u^1_1, \ldots, u^1_n) \in \mathbb{R}^n$, $u^2 = (u^2_1, \ldots, u^2_{n}) \in \mathbb{R}^n$, $u^3 \in \mathbb{R}$ and $u^4 = (u^4_{01}, u^4_{02}, \ldots, u^4_{n, n+1})\in \mathbb{R}^{n+1 \times n+1}$ be the dual variables associated to constraints \eqref{eq:vrp:pstep:rmp:2}--\eqref{eq:vrp:pstep:rmp:3}, respectively. Given a dual solution $\overline{u} = (\overline{u}^1, \ldots, \overline{u}^4)$ of the RMP, where we assume $\overline{u}^1_{0} = \overline{u}^1_{n+1} = 0$, the reduced cost of the column corresponding to a path $(v_0, v_1, \ldots, v_k)$ is given by:
\begin{eqnarray}
\displaystyle rc(\overline{u}) &=& \sum_{j = 0}^{k-1} \left( c_{v_j v_{j+1}} - \overline{u}^{1}_{v_j} - (q_{v_{j+1}} + Q) \overline{u}^{4}_{v_j v_{j+1}} \right)  \nonumber \\ 
& & - \delta({v_0 \neq 0}) \overline{u}^{2}_{v_0} + \delta({v_k \neq n+1}) \overline{u}^{2}_{v_k} - \delta({v_0 = 0}) \overline{u}^{3},  \nonumber 
\end{eqnarray}
where $\delta(C)$ is equal to 1 if condition $C$ is true; $0$, otherwise.

Any feasible path with negative cost corresponds to a path in $\mathcal{S}\backslash\overline{\mathcal{S}}$. This can be used to generate a new column (variable) that has a negative reduced cost and then should be added to the current RMP. After solving the modified RMP a new dual solution is obtained and the process is repeated. If it is not possible to find a path with negative cost, then the current optimal solution of the RMP is also optimal for the linear relaxation of \eqref{eq:vrp:pstep:obj}--\eqref{eq:vrp:pstep:5}. 

Similarly, for the VRPTW, we have the following restricted master problem (RMP):
\begin{eqnarray}
\mbox{min} & \displaystyle  \sum_{s \in \overline{\mathcal{S}}} c^{s} \lambda^{s} &  \label{eq:vrp:pstep:rmp2:obj} \\
\mbox{s.t.} & \displaystyle  \sum_{s \in \overline{\mathcal{S}}} e_{i}^{s} \lambda^{s} = 1, & i = 1, \ldots, n, \label{eq:vrp:pstep:rmp2:2} \\
& \displaystyle  \sum_{s \in \overline{\mathcal{S}}} a_{i}^{s} \lambda^{s} = 0, & i = 1, \ldots, n, \label{eq:vrp:pstep:rmp2:1a} \\
& \displaystyle  \sum_{s \in \overline{\mathcal{S}}} a_{0}^{s} \lambda^{s} \leq K, &  \label{eq:vrp:pstep:rmp2:1b} \\
& \displaystyle \varphi_{i} - \varphi_{j} + \sum_{s \in \overline{\mathcal{S}}_{ij}} (q_j + Q) \lambda^{s} \leq Q, & i = 0, \ldots, n, \ j = 1, \ldots, n+1, \label{eq:vrp:pstep:rmp2:3}\\
& \displaystyle q_{i} \leq \varphi_{i} \leq Q, &  i = 1, \ldots, n, \label{eq:vrp:pstep:rmp2:4} \\
& \displaystyle \omega_{i} - \omega_{j} +  \sum_{s \in \overline{\mathcal{S}}_{ij}} ( s_i + t_{ij} + M_{ij} ) \lambda^s \leq M_{ij}, &  i = 0, \ldots, n, \ j = 1, \ldots, n+1, \label{eq:vrp:pstep:rmp2:4.1} \\
& w^a_i \leq \omega_{i} \leq w^b_i, & i = 0, \ldots, n+1, \label{eq:vrp:pstep:rmp2:4.2} \\
& \displaystyle \lambda_s  \geq 0, & s \in \overline{\mathcal{S}}, \label{eq:vrp:pstep:rmp2:5}
\end{eqnarray}
where $M_{ij}$ is a sufficiently large constant, e.g. $M_{ij} = w^b_i - w^a_j$. As in the CVRP formulation, we have dual variables $u^1$ to $u^4$ corresponding to constraints \eqref{eq:vrp:pstep:rmp2:2} to \eqref{eq:vrp:pstep:rmp2:3}. Additionally, $u^5 = (u^5_{01}, u^5_{02}, \ldots, u^5_{n, n+1})\in \mathbb{R}^{n+1 \times n+1}$ are the dual variables associated to constraints \eqref{eq:vrp:pstep:rmp2:4.1}. Given a dual solution $\overline{u} = (\overline{u}^1, \ldots, \overline{u}^5)$ of the RMP, where we assume $\overline{u}^1_{0} = \overline{u}^1_{n+1} = 0$, the reduced cost of the column corresponding to a route $(v_0, v_1, \ldots, v_k)$ is given by:
\begin{eqnarray}
\displaystyle \bar{z}^{ }_{SP} &=& \sum_{j = 0}^{k-1} \left( c_{v_j v_{j+1}} - (+1) \overline{u}^{1}_{v_j} - (q_{v_{j+1}} + Q) \overline{u}^{4}_{v_j v_{j+1}} - (s_{v_j} + t_{v_{j},v_{j+1}} + M_{v_{j},v_{j+1}}) \overline{u}^{5}_{v_j v_{j+1}} \right) \nonumber \\ 
& & - \delta({v_0 \neq 0}) \overline{u}^{2}_{v_0} + \delta({v_k \neq n+1}) \overline{u}^{2}_{v_k} - \delta({v_0 = 0}) \overline{u}^{3}.  \nonumber 
\end{eqnarray}

Regarding the reduced cost of a column in the set partitioning formulation, we can observe that more dual information is provided for the subproblem in a $p$-step formulation. Indeed, the shadow prices of resources are provided by the RMP and can be used to guide the decision at the subproblem level. In addition, this information can be used with no extra cost in the subproblem, as the duals are defined for pair of nodes and hence can be included as additional costs on the edges of the network. Therefore, even though $p$-step formulations have additional constraints in the MP with respect to SP formulations, we can ensure that they are robust \cite{fukasawa2006}, as the difficulty of solving the subproblem will be the same as in the SP formulation. 

\section{Advantages and disadvantages of the $p$-step formulations}\label{sec:advantages}

We address now a few advantages of the $p$-step formulations with respect to the classical formulation for vehicle routing problems. As mentioned before, different types of capacity constraints and valid inequalities can be incorporated to these formulations. They can be even stated in terms of partial paths instead of arcs, when it leads to stronger versions. Other advantage is that different requirements can be imposed directly to the master problem, which can be very convenient when dealing with rich vehicle routing problems and integrated problems, such as location routing and inventory routing problems \cite{ceselli2009,desaulniers2015}.

At the subproblem level, $p$-step formulations may lead to a better performance, as the label extension is limited by a new resource, the number of steps. Also, the subproblems can be solved in parallel, by splitting the label extension by starting node. Then, we are able to solve $n+1$ subproblems in parallel, which is suitable for the current multi-core CPU machines.

As proposed in this paper, the size of the partial paths in the $p$-step formulation is limited by the number of traversed arcs. Since this can be seen as a resource, other types of resource may be used to limit a path: capacity, timing, etc. For example, we could generate partial paths in which the maximum load is a percentage of the capacity, or the total travel time is less than a percentage of the final time instant. This is somehow a generalization of the bidirectional label-extension \cite{righini2008}, but with the joining of paths done at the master problem. This allows for any type of partitioning in the label-extension, instead of using two partitions only (e.g. several partitions of time).

Of course, $p$-step formulations have a few disadvantages as well. The first one is that the quality of the bound provided by the linear relaxation of a $p$-step formulation depends on $p$. A large $p$ leads to a bound as good as that obtained from the SP formulation, while $p=1$ leads to the weak linear relaxation of the VF formulation. A good strategy would be the use of a \textit{turning point strategy}, in which the value of $p$ is increased during the solution process. Hence, at the turning point, the $p$-step formulation is converted to a $(p+k)$-step formulation, $k \geq 1$, by explicitly combining $p$-step paths to obtain $(p+k)$-step paths.

Another disadvantage of the $p$-step formulation is related to the size of the master problem. Instead of the usual $n$ constraints of the SP formulation, a $p$-step model of the VRPTW has $4n + 2n^2$ constraints, like in the VF formulation. Although the current linear programming solvers are very powerful nowadays, the solution time can be relatively large for $n$ large. Nevertheless, interior point methods can help to overcome this weaknesses, specially if aided by active set strategies for identifying inactive constraints \cite{gondzio2013,munari2013,gondzio2015mpc}.

\section{Conclusions}

In this working paper, we have introduced a general class of formulations for vehicle routing problems, namely the $p$-step formulations. They offer several advantages over classical formulations and seem to be promising in practice. Theoretical results presented in this paper show that the classical formulations are special cases of the $p$-step formulations. Also, the proposed formulation can be put in a column generation scheme that allows more dual information to be sent to the subproblem.

This is still a ongoing research, in its very early stage. A computational implementation of a branch-price-and-cut method for the $p$-step formulation is in course and should be finished soon. Computational results will be reported in a future version of this working paper.

\section*{Acknowledgments}

The authors are thankful to Claudio Contardo and Silvio Araujo for pointing out references \cite{petersen2009} and \cite{fragkos2016} in private communication.

%
%


\end{document}